\documentclass{article}

\usepackage{amssymb,amsthm, amsmath}
\usepackage{hyperref}

\title{Extensions of a result of Elekes and R\'onyai}
\author{Ryan Schwartz \and J\'ozsef Solymosi \and Frank de Zeeuw}

\newtheorem{theorem}{Theorem}[section]

\newtheorem{lemma}[theorem]{Lemma}
\newtheorem{corollary}[theorem]{Corollary}
\newtheorem{conjecture}[theorem]{Conjecture}

\begin{document}
\maketitle

\begin{abstract}
Many problems in combinatorial geometry can be formulated in terms of curves or 
surfaces containing many points of a cartesian product.  In 2000, Elekes and 
R\'onyai proved that if the graph of a polynomial contains $cn^2$ points of an 
$n\times n\times n$ cartesian product in $\mathbb{R}^3$, then the polynomial has 
the form $f(x,y)=g(k(x)+l(y))$ or $f(x,y)=g(k(x)l(y))$.  They used this to prove 
a conjecture of Purdy which states that given two lines in $\mathbb{R}^2$ and 
$n$ points on each line, if the number of distinct distances between pairs of 
points, one on each line, is at most $cn$, then the lines are parallel or 
orthogonal.  We extend the Elekes-R\'onyai Theorem to a less symmetric cartesian 
product.  We also extend the Elekes-R\'onyai Theorem to one dimension higher on 
an $n\times n\times n\times n$ cartesian product and an asymmetric cartesian 
product.  We give a proof of a variation of Purdy's conjecture with fewer points 
on one of the lines.  We finish with a lower bound for our main result in one 
dimension higher with asymmetric cartesian product,
showing that it is near-optimal.  
\end{abstract}

\section{Introduction} \label{sec:intro}

\subsection{Background}
We are interested in polynomials on finite cartesian products, for instance of 
the form $f(x,y)\in\mathbb{R}[x,y]$ on $A\times B$, with $A,B\subset \mathbb{R}$ 
and $|A| = |B| = n$.
We will focus on the question of 
how small the image $f(A,B)$ can be in terms of $n$.

For two basic examples, $x+y$ and $xy$, the image can be as small as $cn$, 
if $A$ and $B$ are chosen appropriately. 
For $f(x,y) = x+y$ one can take $A = B = [1,n]$ (or any other arithmetic 
progression of length $n$), so that $f(A,B) = A+B =[2,2n]$;
for $f(x,y) = xy$ one can take a geometric progression like
$A = B = \{2^1,2^2, \ldots, 2^n\}$, so that 
$f(A,B) = A\cdot B = \{2^2,2^3,\ldots, 2^{2n}\}$. 
Similar small images can be obtained for polynomials of the form 
$f(x,y) = g(k(x) + l(y))$, for nonconstant polynomials $g,k,l$, 
by taking $A$ so that $k(A) \subset [1,n]$,
 and $B$ so that $l(B) \subset [1,n]$.
A similar idea works for $f(x,y) = g(k(x)\cdot l(y))$.

For convenience, we will formulate the problem slightly differently:
we consider the surface $z = f(x,y)$ in $\mathbb{R}^3$ 
and its intersection with a cartesian product $A\times B \times C$, 
with $|A| = |B| = |C| = n$.
Then the image of $f$ is small if and only if the intersection is large; 
for instance, $z = x+y$ has intersection with $[1,n]^3$ 
of size at least $\frac{1}{4}n^2$.

In 2000, Elekes and R\'onyai \cite{Elek00}
proved the following converse of the above observations.
\begin{theorem}[Elekes-R\'onyai Theorem] \label{thm:ER}
For every $c>0$ and positive integer $d$ there exists $n_0 = n_0(c,d)$ with the 
following property.\\ Let $f(x,y)$ be a polynomial of degree $d$ in 
$\mathbb{R}[x,y]$ such that for an $n>n_0$ the graph $z = f(x,y)$ 
contains $c n^2$ points of $A\times B\times C$, where $A,B,C\subset \mathbb{R}$ 
have size $n$. Then either
$$f(x,y) = g(k(x) + l(y)),~~~\mathrm{or}~~~ f(x,y) = g(k(x)\cdot l(y)),$$
where $g,k,l\in \mathbb{R}[t]$.
\end{theorem}

In fact, they proved that the same is true for rational functions,
if one allows a third special form $f(x,y) = g((k(x)+l(y))/(1-k(x)l(y)))$.
Elekes and Szab\'o \cite{ElSz, Elek09} were able to extend this theorem to 
implicit surfaces $F(x,y,z) = 0$, and also showed
that the surface need only contain
$n^{2-\gamma}$ points of the cartesian product for the conclusion to hold,
for some absolute 'gap' $\gamma>0$.

Elekes and R\'onyai used their result to prove a famous conjecture of Purdy.
It says that given two lines in $\mathbb{R}^2$ and $n$ points on each 
line, if the number of distinct distances between pairs of points, one on each 
line, is $cn$ for some $c>0$, then the lines are parallel or orthogonal.  
Elekes \cite{Elek99} also proved a 'gap version', 
only requiring the number of distances to be less than $cn^{5/4}$.
For details and a variation of Purdy's conjecture, using our results below,
see Section~\ref{subsec:purdy}.

See \cite{Elek02, Mato02, Mato11} for more detail and some related problems.

\subsection{Results}
In this paper we prove a number of extensions of Theorem~\ref{thm:ER}.
We extend the result to one dimension higher, 
to asymmetric cartesian products, and to both at the same time.
The proofs are based on the proof of Theorem~\ref{thm:ER} by Elekes and R\'onyai.

First we consider a less symmetric cartesian product.

\begin{theorem} \label{thm:ER2}
For every $c>0$ and positive integer $d$ there exist $n_0 = n_0(c,d)$ 
and $\tilde{c} = \tilde{c}(c,d)$ with the following property.\\
 Let $f(x,y)$ be a polynomial of degree 
$d$ in $\mathbb{R}[x,y]$ such that for an $n>n_0$ the graph $z = f(x,y)$ 
contains $c n^{11/6}$ points of $A\times B\times C$, 
where $A,B,C\subset \mathbb{R}$ and $|A| = n, |B| = \tilde{c}n^{5/6}$, and $|C| = n$.
Then either
$$f(x,y) = g(k(x) + l(y)),~~~\mathrm{or}~~~ f(x,y) = g(k(x)\cdot l(y)),$$
where $g,k,l\in \mathbb{R}[t]$.
\end{theorem}

Using a recent result of Amirkhanyan, Bush, Croot and Pryby \cite{Amir11}
regarding a conjecture of Solymosi 
about the number of lines in general position that can be rich on a cartesian product
(see Section \ref{subsec:linelemmas}),
we get the following theorem.

\begin{theorem} \label{thm:ERbest}
For every $c>0$ and positive integer $d$ there exists $n_0 = n_0(c,d)$ with the 
following property.\\  
Let $f(x,y)$ be a polynomial of degree $d$ in 
$\mathbb{R}[x,y]$ such that for an $n>n_0$ the graph $z = f(x,y)$ 
contains $c n^{3/2 + \varepsilon}$ points of $A\times B\times C$, where 
$A,B,C\subset \mathbb{R}$ and $|A| = n, |B| = n^{1/2 + \varepsilon}$ 
with $\varepsilon > 0$, and $|C| = n$.
Then either
$$f(x,y) = g(k(x) + l(y)),~~~\mathrm{or}~~~ f(x,y) = g(k(x)\cdot l(y)),$$
where $g,k,l\in \mathbb{R}[t]$.
\end{theorem}

We also extend the Elekes-R\'onyai Theorem to cartesian products of one dimension higher,
i.e. to polynomials with one more variable.
\begin{theorem} \label{thm:main}
For every $c>0$ and positive integer $d$ there exists $n_0 = n_0(c,d)$ with the 
following property.\\  
Let $f(x,y,z)$ be a polynomial of degree $d$ in $\mathbb{R}[x,y,z]$ such that 
for an $n>n_0$ the graph $w = f(x,y,z)$ 
contains $c n^3$ points of $A\times B\times C\times D$, 
where $A,B,C,D\subset \mathbb{R}$ have size $n$.
Then either
$$f(x,y,z) = g(k(x) + l(y) + m(z)),~~~\mathrm{or}~~~ f(x,y,z) = g(k(x)\cdot 
l(y)\cdot m(z)),$$
where $g,k,l,m\in \mathbb{R}[t]$.
\end{theorem}

We can also prove a higher-dimensional version with a less symmetric cartesian 
product.
\begin{theorem} \label{thm:ER3}
For every $c>0$ and positive integer $d$ there exists $n_0 = n_0(c,d)$ with the 
following property.\\
Let $f(x,y,z)$ be a polynomial of degree $d$ in 
$\mathbb{R}[x,y,z]$ such that for an $n>n_0$ the graph $w = f(x,y,z)$ 
contains $c n^{8/3+2\varepsilon}$ points of $A\times B\times C\times D$, 
where $A,B,C,D\subset \mathbb{R}$ and $|A| = n$, 
$|B|=|C| = n^{5/6+\varepsilon}$ with $\varepsilon > 0$, 
and $|D| = n$.
Then either
$$f(x,y,z) = g(k(x) + l(y) + m(z)),~~~\mathrm{or}~~~ f(x,y,z) = g(k(x)\cdot 
l(y)\cdot m(z)),$$
where $g,k,l,m\in \mathbb{R}[t]$.
\end{theorem}

And using the abovementioned result of Amirkhanyan et al.~we get the following:
\begin{theorem} \label{thm:ER4}
Given $c>0$ and $d$ a positive integer there exists $n_0 = n_0(c,d)$ with the 
following property. \\  Let $f(x,y,z)$ be a polynomial of degree $d$ in 
$\mathbb{R}[x,y,z]$ such that for an $n>n_0$ the graph $w = f(x,y,z)$ contains 
$c n^{2 + 2\varepsilon}$ points of $A\times B\times C\times D$, where 
$A,B,C,D\subset \mathbb{R}$ and $|A| = n$,
$|B|=|C| = n^{1/2 + \varepsilon}$ with $\varepsilon > 0$, 
and $|D| = n$.
Then either
$$f(x,y,z) = g(k(x) + l(y) + m(z)),
~~~\mathrm{or}~~~ f(x,y,z) = g(k(x)\cdot l(y)\cdot m(z)),$$
where $g,k,l,m\in \mathbb{R}[t]$.
\end{theorem}

In Section~\ref{subsec:parab} we will give an example of a polynomial $f(x,y,z)$ 
whose graph contains $cn^2$ points of $A\times B\times C\times D$, 
where $|A|=|D|=n$ and $|B|=|C|=c'n^{1/2}$,
but $f$ does not have the required additive or multiplicative form of Theorem~\ref{thm:ER4}.
This shows that Theorem~\ref{thm:ER4} is near-optimal.

Note that as for the two-variable case, 
the converses of Theorems~\ref{thm:main}--\ref{thm:ER4} all hold 
for some appropriate cartesian products.  
Specifically, if $f(x,y,z) = g(k(x)+l(y)+m(z))$,
one can choose $A$, $B$, and $C$ so that $k(x)$, $l(y)$, and $m(z)$
have values in the same arithmetic progression.
A similar construction works for the product case.

Theorems~\ref{thm:ER}, \ref{thm:ER2}, \ref{thm:main} and \ref{thm:ER3} would all hold 
if we consider functions over $\mathbb{C}$ instead of $\mathbb{R}$,
but we will restrict ourselves to $\mathbb{R}$ here.
The proofs could be extended to $|B|\neq |C|$, at some cost to the exponents.
It also seems possible to generalize our proofs to polynomials with even more variables.

In Section \ref{subsec:outline} we give a short outline of the proof of the 
Elekes-R\'onyai Theorem,
which provides a template for our subsequent proofs.
Section \ref{sec:prelim} contains a number of concepts and results 
required throughout our proofs.
In Section~\ref{sec:erLess} we give the proofs of Theorems \ref{thm:ER2} and \ref{thm:ERbest}, 
while Section \ref{sec:er4d} contains the proofs of Theorems \ref{thm:main}, 
\ref{thm:ER3}, and \ref{thm:ER4}.  
In Section~\ref{sec:apps} we give an extension of the conjecture of Purdy 
and an example showing the near-optimality of Theorem~\ref{thm:ER4}.

\subsection{Outline of proofs}\label{subsec:outline}
The following is an outline of the proof 
that Elekes and Ronyai gave in \cite{Elek00} of Theorem \ref{thm:ER}.
Our theorems are obtained by adjusting this proof to three-variable $f$, 
and by using improved Line Lemmas (see Section \ref{subsec:linelemmas}) 
to get the asymmetric versions.
All functions below are polynomials, 
and we repeatedly recycle the positive constant $c$.

We split up the surface $z = f(x,y)$ into the $n$ curves
$$z = f_i(x) = f(x, b_i),$$
for each of the $b_i\in B$. We wish to decompose a $cn$-sized subset of the $f_i$ as
$$f_i(x) = (p\circ \varphi_i\circ k)(x) = p(a_i k(x) + b_i),$$
where $\varphi_i$ is linear and $p$ and $k$ are independent of $i$.

Then the $cn$ lines $u = \varphi_i(t) = a_i t + b_i$ will also be 
$cn$-rich on an $n\times n$ cartesian product.
For such sets of lines we have various lemmas (\ref{lem:line}--\ref{cor:croot})
that say that a $cn$-sized subset of them must be all parallel or all concurrent.

Given $cn$ such decompositions with the lines $\varphi_i$ all parallel,
we can write $f(x,y) = p(a k(x) + b_i)$, 
and then conclude by an algebraic argument that there exists an $l(y)$ 
such that $f(x,y) = p(k(x) + l(y))$.
If $cn$ of the lines are concurrent,
we can write $f(x,y) = p(a_i\cdot (k(x)+b))$, 
and then conclude that $f(x,y) = p(k(x)\cdot l(y))$.

To find the above decomposition of the $f_i$, 
we first remove their common inner functions (polynomials $\mu$ such that $f_i = 
\lambda_i\circ\mu$) up to linear equivalence.
We can do this because the number of decompositions up to linear equivalence
of a polynomial of degree $d$ depends only on $d$ (Lemma \ref{lem:decomp}), 
so for large enough $n$ there must be a $cn$-sized subset of the $f_i$ 
that all have the same inner function of maximal degree.
This maximal inner function will be the $k$ above, 
and we remove it by writing $f_i = \widehat{f}_i\circ k$.
Then we have a subset of $\widehat{f}_i$ with the property that
if $\widehat{f}_i = \mu_i\circ \lambda$ and $\widehat{f}_j = \mu_j\circ \lambda$, 
then $\lambda$ must be linear.

Now we combine pairs $\widehat{f}_i, \widehat{f}_j$ into new curves
$$\gamma_{ij}(t) = (\widehat{f}_i(t), \widehat{f}_j(t)).$$
We observe that these $\gamma_{ij}$ are $cn$-rich 
on an $n\times n$ cartesian product, 
and that we have $cn^2$ of them.
But by a theorem of Pach and Sharir (Lemma \ref{lem:curve}), 
such a set of rich curves can have size at most $c'n$. 

This is not a contradiction: 
many of these $\gamma_{ij}$ may coincide as sets in $\mathbb{R}^2$.
But if for instance $\gamma_{ij}$ and $\gamma_{kl}$ coincide,
then by some algebra (Lemma \ref{lem:repar}) they must be reparametrizations 
of the same curve $(p(t), q(t))$, which means that we can write
$$\begin{array}{cc}
\widehat{f}_i = p\circ \varphi, & \widehat{f}_j = q\circ \varphi,\\
\widehat{f}_k = p\circ \phi, & \widehat{f}_l = q\circ \phi.
\end{array}$$
Since we already removed all nonlinear common inner polynomials,
$\varphi$ must be linear.
If we have enough such decompositions, 
we can ensure that they all have the form $f_i = p\circ \varphi_i$
for the same $p$.
This give us the desired decompositions
$$f_i = \widehat{f}_i\circ  k = p\circ \varphi_i\circ k.$$

\section{Preliminaries} \label{sec:prelim}

\subsection{Discrete geometry}
We will make frequent use of the following well-known theorem, 
first proved in \cite{Szem83}.
We say that a line (or any other curve) is \textit{$k$-rich} on a point set $\mathcal{P}$ 
if it contains at least $k$ points of $\mathcal{P}$.


\begin{theorem}[Szemer\'edi-Trotter Theorem] \label{thm:szem}
There exists a constant $C_{ST}>0$ such that 
given a set $\mathcal{P}$ of $n$ points in $\mathbb{R}^2$, 
the number of lines $k$-rich on $\mathcal{P}$
is at most $C_{ST}\cdot(n^2/k^3 + n/k)$.
\end{theorem}

This theorem was generalized by Pach and Sharir \cite{Pach90, Pach98}
to continuous real planar curves without self-intersection.
We will use the following corollary for algebraic curves, 
which follows quite easily since algebraic curves (of bounded degree) can be 
split up into a small number (depending on the degree) of curves without 
self-intersection.
For details see Elekes and R\'onyai \cite{Elek00}.

\begin{lemma}[Curve Lemma] \label{lem:curve}
Given $c>0$ and a positive integer $d$, 
there exist $C_{CL} = C_{CL}(c,d)$ and $n_0 = n_0(c,d)$ such that the following holds.\\
Given a set of $m$ irreducible real algebraic curves of degree $\leq d$ 
that are $cn$-rich on $A$, where $A\subset \mathbb{R}^2$ and $|A|\leq n^2$,
then for all $n>n_0$ we have
\[m\leq C_{CL}\cdot n.\] 
\end{lemma}

\subsection{Line lemmas}\label{subsec:linelemmas}

In the proof of Theorem \ref{thm:ER} by Elekes and R\'onyai,
an important ingredient was the following result of 
Elekes \cite{Elek97} about lines containing many points from a cartesian product.

\begin{lemma}[Line Lemma] \label{lem:line}
Suppose $A,B\subset \mathbb{R}$ and $|A| = |B| = n$.  
For all $c_1,c_2>0$ there exists $C_{LL}>0$, independent of $n$, 
such that if $m$ lines in $\mathbb{R}^2$ are $c_1n$-rich on $A\times B$,
with no $c_2n$ of the lines all parallel or all concurrent,
then 
$$m < C_{LL}\cdot n.$$
\end{lemma}

We prove a generalization that will be crucial in Section~\ref{sec:erLess}.
The proof is at the end of this section, and is modelled on that of Elekes.

\begin{lemma}[Generalized Line Lemma] \label{lem:genLine}
Suppose $A,B\subset \mathbb{R}$ and $|A| = |B| = n$.  
For all $c_1,c_2>0$, and $\beta \ge 0$ there exists $C_{GLL}>0$, 
independent of $n$, such that 
if $m$ lines in $\mathbb{R}^2$ are $c_1n$-rich on $A\times B$, 
with no $c_2n^{\beta}$ concurrent or parallel,
then 
$$m < C_{GLL}\cdot n^{2/3 + \beta/3}.$$
\end{lemma}

A collection of lines in $\mathbb{R}^2$ is said to be in \emph{general position} 
if no two lines are parallel and no three lines are concurrent.
The second author conjectured the following extension of the above result.  For 
details see \cite{Elek02}.
\begin{conjecture}
Suppose $A,B\subset \mathbb{R}$ and $|A|=|B|=n$.  For all $c>0$ there exists
$C_S>0$ such that if $m$ lines in general position are $cn$-rich on $A\times B$ 
then $m<C$.

\end{conjecture}

The following result of Amirkhanyan et al.\cite{Amir11}~
is related to the above conjecture.

\begin{theorem} \label{thm:croot}
   For every $\varepsilon>0$ there exists $\delta>0$ such that given 
   $n^{\varepsilon}$ lines in $\mathbb{R}^2$ in general position, they cannot 
   all be $n^{1-\delta}$-rich on $A\times B$, where $|A|=|B|=n$.
\end{theorem}
Thus if a collection of lines $\mathcal{L}$ in general position 
is $cn$-rich on $A\times B$ 
then $|\mathcal{L}| < n^{\varepsilon}$ for any $\varepsilon > 0$.  
We will use it in the form of the following corollary.

\begin{corollary} \label{cor:croot}
   If $m$ lines in $\mathbb{R}^2$  are $cn$-rich on $A\times B$, 
   with $|A|=|B|=n$, such that no $p$ are parallel and no $q$ are concurrent, 
   then \[m \le (p+q)n^{\varepsilon}\] for every $\varepsilon > 0$.

\end{corollary}
\begin{proof}
We show that the collection of lines contains at least 
$k=\sqrt{m}/\sqrt{2(p+q)}$ lines in general position.

We pick any line, and then successively choose a new line that is not 
parallel to any of the previously chosen lines, and does not go through the 
intersection point of any pair of them.  
If we have chosen $k$ such lines, 
then there are $k$ slopes we may not choose, which excludes less than $pk$ lines.  
And there are at most $\binom{k}{2}$ intersection points that we must avoid, 
so since there are less than $q$ lines concurrent at a point, 
this excludes less than $q\binom{k}{2}$ lines.  
Hence we can continue in this way at least until 
$m \le q\binom{k}{2}+pk+k < (p+q)k^2$, 
so we can get $k\ge \sqrt{m}/\sqrt{p+q}$ lines in general position.  
These lines are $cn$-rich on $A\times B$.  
Thus $k\le n^{\varepsilon'}$ for every $\varepsilon' > 0$.  
This gives $m\le (p+q)n^{\varepsilon}$ for every $\varepsilon>0$.
\end{proof}

We begin the proof of Lemma~\ref{lem:genLine}.  We will use the dual of a 
theorem of Beck \cite{Beck83},
which roughly states that given a collection of points, 
either ``many'' of the points are on the same line, 
or pairs of the points determine ``many'' distinct lines.

\begin{theorem}[Dual of Beck's Theorem] \label{thm:beck}
There exists $C_{BT}>0$ such that, given $N$ lines in $\mathbb{R}^2$, 
either $C_{BT}N$ lines are concurrent 
or the lines determine $C_{BT}N^2$ distinct pairwise intersection points.
\end{theorem}

\begin{proof}[Proof of Lemma~\ref{lem:genLine}]
Let $L$ be the set of lines, $|L| = m = cn^{\alpha}$, 
so we will show that we can take $\alpha=2/3+\beta/3$ and $c$ some constant.  
For every pair $(\ell_i,\ell_j)\in L^2$ we define the linear functions 
$\gamma_{ij} = \ell_i\circ\ell_j^{-1}$ and 
$\Gamma_{ij} = \ell_j^{-1}\circ\ell_i$.

\bigskip
First we will prove that large subsets of the $\gamma_{ij}$ and $\Gamma_{ij}$ are also rich.
Consider the tripartite graph $H$ with vertex sets $A\cup L\cup B$.  
Given $a\in A$ and $\ell \in L$, $(\ell, a)$ is an edge in $H$ if $\ell(a) \in B$.  
Similarly, given $\ell \in L$ and $b\in B$, $(\ell, b)$ is an edge if $\ell^{-1}(b)\in A$. 
Given $\ell\in L$, let $\deg_A(\ell)$ be the number of edges between $\ell$ and $A$ and 
$\deg_B(\ell)$ the number of edges between $l$ and $B$.  
Since the lines in $L$ are $c_1n$-rich on $A\times B$, 
we have $\deg_A(\ell) \ge c_1n$ and $\deg_B(\ell) \ge c_1n$ for each $\ell \in L$.  
Thus we have at least $cc_1n^{1+\alpha}$ edges between $A$ and $L$ 
and at least $cc_1n^{1+\alpha}$ edges between $B$ and $L$.

We will count cycles of length four in $H$ with one vertex in $A$ and one vertex in $B$.  
Every such $C_4$ gives a point in $B\times B$ on $\gamma_{ij}$ 
and a point in $A\times A$ on $\Gamma_{ij}$ for some pair $(i,j)$.
The number of paths of length two with one endpoint in $A$ and the other in $B$ 
is at least 
\[\#P_2 = \sum_{\ell\in L}\deg_A(\ell)\deg_B(\ell) \ge c_1^2n^{2+\alpha}.\]  
Let $p_{a,b}$ be the number of paths of length two between $a\in A$ and $b\in B$.  
Then \[\#P_2 = \sum_{a\in A, b\in B}p_{a,b}.\]
Now, by Jensen's Inequality, 
the number of $C_4$'s we are looking for is
\[\#C_4 = \sum_{a\in A, b\in B}\binom{p_{a,b}}{2} \ge 
|A\times B|\binom{\#P_2/|A\times B|}{2} \ge \frac{c_1^4n^{2+2\alpha}}{4}.\]
Suppose there are fewer than $(c_1^4/8)n^{2\alpha}$ pairs 
 $(\ell_i,\ell_j)\in L^2$ with at least $(c_1^4/8)n^2$ $C_4$'s between them.  
Then $H$ would have fewer than $(c_1^4/4)n^{2+2\alpha}$ $C_4$'s, a contradiction.  
Thus, setting $c_3=c_1^4/8$, we have at least $c_3n^{2\alpha}$ pairs $(i,j)$
for which $\gamma_{ij}$ and $\Gamma_{ij}$ are $c_3n$-rich 
on $B\times B$ and $A\times A$ respectively.

\bigskip
Next we define a different graph $G'$ and analyze it.
The vertex sets of $G'$ consist of those $\gamma_{ij}$ that are $c_3n$-rich on $B\times B$ 
and those $\Gamma_{kl}$ that are $c_3n$-rich on $A\times A$.  
If $\gamma_{ij}$ and $\gamma_{kl}$ coincide as point sets, we consider them as the same vertex.
Similarly we identify any coinciding $\Gamma_{ij}$ and $\Gamma_{kl}$, 
but we do not identify $\gamma_{ij}$ and $\Gamma_{kl}$ should they coincide.
We place an edge between $\gamma_{ij}$ and $\Gamma_{ij}$ for each pair $(i,j)$,
which means the graph is bipartite.

The graph may contain multiple edges, 
if we have $\ell_i, \ell_j, \ell_k, \ell_l\in L$ 
such that both $\gamma_{ij} = \gamma_{kl}$ and $\Gamma_{ij} = \Gamma_{kl}$.
But this implies that the four lines are concurrent: 
If $\ell_i$ and $\ell_j$ intersect in $(u,v)$, 
then $v=\ell_i\circ\ell_j^{-1}(v)=\ell_k\circ\ell_l^{-1}(v)$ 
 and $u=\ell_j^{-1}\circ\ell_i(u)=\ell_l^{-1}\circ\ell_k(u)$, 
so $(u,v)$ is also the intersection point of $\ell_k$ and $\ell_l$.

But with Beck's Theorem we can get a subgraph without multiple edges.
We will assume that $\beta < \alpha$, and check it at the end of the proof.
Then fewer than $c_2n^{\beta} < C_{BT}cn^{\alpha} = C_{BT}|L|$ lines are concurrent, 
so by Theorem~\ref{thm:beck}, 
the lines determine $C_{BT}n^{2\alpha}$ distinct intersection points.
The corresponding lines span a subgraph $G''$ without multiple edges, 
and at least $C_{BT}n^{2\alpha}$ edges. 

By the Szemer\'edi-Trotter Theorem, since all vertices are $c_3n$-rich lines, 
the number of vertices is at most $\le c_4n$ for some constant $c_4>0$.  
The average degree in $G''$ is then $\geq (C_{BT}/c_4)n^{2\alpha-1}$.  
Thus $G''$ contains a connected component $H$
containing at least $(C_{BT}/c_4)n^{2\alpha-1}$ vertices 
and at least $\frac{1}{2}(C_{BT}/c_4)^2n^{4\alpha-2}$ edges.  

Note that each $\gamma_{ij}$ and $\Gamma_{ij}$ have the same slope,
so every vertex in $H$ is a line with the same slope.  
If there are more than $cc_2n^{\alpha+\beta}$ edges in $H$ then 
we would have $\gamma_{ij_1}, \gamma_{ij_2},\dots, \gamma_{ij_k}$ vertices 
in this component with $k \ge c_2n^{\beta}$.  
This implies that the lines $\ell_{j_1}, \ell_{j_2}, \dots, \ell_{j_k}$ are all parallel,
which is a contradiction.  
So we have \[\frac{C_{BT}^2}{2c_4^2}n^{4\alpha-2} < cc_2n^{\alpha+\beta}.\]  
From this we see that we can choose $\alpha = 2/3 + \beta/3$ and $c> C_{BT}^2/(2c_2c_4^2)$.
\end{proof}

\subsection{Algebra and graph theory}
In the proofs of our higher-dimensional versions of the Elekes-Ronyai Theorem, 
we will need the following generalization of the fact that if a degree-$d$ 
polynomial of one variable has $d+1$ or more roots, then the polynomial is 
identically zero.

\begin{lemma}[Vanishing Lemma] \label{lem:vanish}
Let $K$ be a field, $F(y,z)\in K[y,z]$ with $\deg F = d$, 
and $B,C\subset K$ with $|B| = |C| = m$.

If $F(y_i,z_j) =0$ for $2dm$ of the pairs $(y_i,z_j)\in B\times C$,
then $F(y,z) = 0$.
\end{lemma}
\begin{proof}
 There must be $d+1$ columns with $d+1$ zeroes, 
i.e. $d+1$ $y_i$ such that for each there are $d+1$ $z_j$ with $F(y_i,z_j) = 0$.
Indeed, after finding $d$ such columns and removing them, 
we are left with at least $2md - md = md$ zeroes, 
distributed over the $m-d$ remaining columns, 
so there must be another column with $d+1$ zeroes.
(Note that the exact bound is $2md - d^2+1$, 
but the simpler formula suffices for us.)

Since a nonzero polynomial in one variable of degree at most $d$ 
can have no more than $d$ roots, 
we have $F(y_i,z) = 0$ for each of the $y_i$ with $d+1$ zeroes.
Let $L = K(z)$ and define the polynomial $G(y)\in L[y]$ by $G(y) = F(y,z)$, 
so also $\deg G \leq d$.  
We have $G(y_i) =0$ for the $d+1$ different $y_i$, 
which implies by the same fact that $G(y) = 0$, hence also $F(y,z) = 0$.  
 
\end{proof}




We will also need the following three algebraic lemmas, 
which appear with proofs in \cite{Elek00}.
Let $K$ be a field.  
We call two decompositions $f(x) = \varphi_1(\psi_1(x))$ and $f(x) = \varphi_2(\psi_2(x))$ 
of a polynomial $f \in K[x]$ into polynomials from $K[x]$ 
\emph{equivalent} if $\psi_1(t) = a\psi_2(t)+b$ for some $a,b \in K$.

\begin{lemma} \label{lem:decomp}
Let $K$ be a field.  
Then no $f \in K[x]$ can have more than $2^d$ non-equivalent decompositions,
 where $d=\deg f$.
\end{lemma}

\begin{lemma} \label{lem:alg}
   \text{ }\begin{enumerate}
      \item Let $E$ be a field, $\varphi \in E[x]$ a polynomial of degree $d>0$.  
         Then every $F \in E[x]$ can be written in the form \[F = a_0 + a_1 x + 
         \dots + a_{d-1} x^{d-1},\] where $a_i \in E(\varphi),$ in a unique way.
      \item Suppose further that $E=L(y)$ is a rational function field over some 
         field $L$, and $\varphi \in L[x]$.  Let $m$ be the degree of $F$ in 
         $y$.  Then the degree of $a_i$ in $y$ is at most $m(d+1)$.  (Here $a_i$ 
         is viewed as a polynomial of $\varphi$ and $y$ over 
         $L$.)
   \end{enumerate}
\end{lemma}

\begin{lemma}[Reparametrization Lemma] \label{lem:repar}
Suppose that two parametric curves $(f_1(t), g_1(t))$ and $(f_2(t), g_2(t))$ 
coincide as sets, with $f_i,g_i\in K[t]$ for a field $K$.
Then there are $p,q,\varphi_1, \varphi_2\in K[t]$ such that
$$\begin{array}{cc}
f_1 = p\circ \varphi_1, & g_1 = q\circ \varphi_1,\\
f_2 = p\circ \varphi_2, & g_2 = q\circ \varphi_2.
\end{array}$$
\end{lemma}

Finally, we need the following graph-theoretic lemma, also proved in \cite{Elek00}.
\begin{lemma}[Graph Lemma]\label{lem:graph}
For every $c$ and $k$ there is a $C_{GL} = C_{GL}(c,k)$ with the following property.\\
If a graph has $N$ vertices and $cN^2$ edges, and the edges are colored so that at most $k$ colors meet at each vertex,
then it has a monochromatic subgraph with $C_{GL}N^2$ edges.
\end{lemma}

\section{Proof of Theorems~\ref{thm:ER2}~and~\ref{thm:ERbest}} 
\label{sec:erLess}

Suppose $z = f(x,y)$ contains $c n^{\alpha + 1}$ points of $A\times B\times C$,
where $|A| = |C| = n$ and $|B| = \tilde{c}n^\alpha$; 
$\tilde{c}$ and $\alpha$ will be determined later.
Throughout we will use $d = \deg f$. All functions will be polynomials.
  
\subsection{Constructing $\widehat{f}_i$}
For each of the $\tilde{c}n^\alpha$ $b_i\in B$ define 
$$f_i(x) = f(x, b_i).$$
Then each $f_i$ is a polynomial in $\mathbb{R}[x]$ of degree at most $d$.

\begin{lemma}\label{lem:nottoomanysame2D}
 If at least $d+1$ of the $f_i$ are identical, then $f(x,y) = q(x)$.\\
In particular, the conclusion of Theorems~\ref{thm:ER2}~and~\ref{thm:ERbest} holds.
\end{lemma}
\begin{proof}
Suppose that $f_i(x) = q(x)$ at least $d+1$ times.  
Then considering $F(y) = f(x,y) - q(x)$ as 
a polynomial in $y$ over the field $\mathbb{R}(x)$, 
we have $F(y)$ vanishing $d+1$ times, so $F(y) = 0$ identically.
\end{proof}\noindent
{\bf Assumption:} Throughout the rest of this section we will assume that 
at most $d$ of the $f_i$ are identical.
\bigskip

Let $c_1 = \min(c/2, \tilde{c}/2)$.  Then at least $c_1n^{\alpha}$ of the $f_i$ are 
$c_1n$-rich on $A\times C$.  Otherwise $z=f(x,y)$ would contain fewer than 
$cn^{\alpha + 1}$ points of $A\times B\times C$.

We construct a graph $G$ with the $c_1n^{\alpha}$ $c_1n$-rich $f_i$ as vertices and 
edge set $E$ consisting of all pairs $(f_i,f_j)$.
\begin{lemma}
 There is a subgraph of $G$ with edge set $\widehat{E}\subset E$ of size 
 $|\widehat{E}|\geq c_2n^{2\alpha}$, such that the following holds.
There is a polynomial $k(x)$ such that for all $(f_i, f_j)\in \widehat{E}$ we 
can write
$$f_i = \widehat{f}_i\circ k,$$ $$f_j = \widehat{f}_j\circ k,$$ and 
$\widehat{f}_i$ and $\widehat{f}_j$ share no non-linear common inner function.\\
The $\widehat{f}_i$ are also $c_1n$-rich on $k(A)\times C$.
\end{lemma}
\begin{proof}
Color each edge $(f_i, f_j)$ of $G$ with the equivalence class of a common inner 
function $\varphi$ of maximum degree, i.e. $f_i(x) = g(\varphi(x))$ and $f_j(x) 
= h(\varphi(x))$, and no such $\varphi$ of higher degree exists; two such 
inner functions $\varphi, \phi$ are equivalent if $\phi(x) = a\varphi(x) + b$.

By Lemma~\ref{lem:decomp}, at every vertex there are at most $2^d$ colors, 
so by Graph Lemma \ref{lem:graph}, with $N = c_1n^{\alpha}$, 
there is a monochromatic subgraph with $C_{GL}N^2 = C_{GL}c_1^2n^{2\alpha}$ edges.
We take $\widehat{E}$ to be the edge set of this subgraph.
This means that all the $f_i$ involved in this subgraph have a common inner 
function $k(x)$ 
(actually up to equivalence, 
but by modifying the $\widehat{f}_i$ that is easily overcome),
and no pair corresponding to an edge of $\widehat{E}$
has a common inner function of higher degree.
That allows us to define the $\widehat{f}_i$ as in the theorem; they must be rich 
since otherwise the $f_i$ could not be rich.
\end{proof}

\subsection{Constructing $\gamma_{ij}$}
For the $c_2n^{2\alpha}$ pairs 
$\widehat{f}_i, \widehat{f}_j$ for which $(f_i,f_j)\in \widehat{E}$,
we construct the curves
$$\gamma_{ij}(t) = \left(\widehat{f}_i(t), \widehat{f}_j(t)\right).$$
      
\begin{lemma}\label{lem:gammas}\text{ }
   \begin{enumerate}
      \item At least $c_3n^{2\alpha}$ $\gamma_{ij}$ are $c_3n$-rich on $C\times C$.
      \item Each $\gamma_{ij}$ is an irreducible algebraic curve of degree at most $2d$.
   \end{enumerate}
\end{lemma}
\begin{proof}
\begin{enumerate}
\item We define a bipartite graph with vertex set $k(A) \cup \{\widehat{f}_i\}$, 
and we connect $t\in k(A)$ with $\widehat{f}_i$ if $\widehat{f}_i(t)\in C$.  
Since $|\widehat{E}|\geq c_2n^{2\alpha}$, 
the number of $\widehat{f}_i$ is at least $\sqrt{c_2}n^{\alpha}$, 
each of them $c_2n$-rich,
so the bipartite graph has $m \ge c_2^{3/2} n^{\alpha +1}$ edges.  
We count the paths of length two between different $\widehat{f}_i$'s,
using that $k(A)\leq n$:
$$\#P_2 = \sum_{t\in k(A)} \binom{\deg(t)}{2}
\ge |k(A)|\binom{m/|k(A)|}{2}\ge c' n^{2\alpha +1}.$$
Hence at least $c''n^{2\alpha}$ pairs $(\widehat{f}_i,\widehat{f}_j)$ share 
$c''n$ common neighbors $t$ in this graph.
In other words, $c''n^{2\alpha}$ of the $\gamma_{ij}$ 
have a point in $C\times C$ for $c''n$ different $t$.  
         
It is possible that different $t$ give the same point $\gamma_{ij}(t)$, 
so these $\gamma_{ij}$ could have fewer than $c''n$ points in $C\times C$.
However, because $\deg \widehat{f}_i \le d$, 
this can happen for at most $d$ different $t$ at a time, 
so each $\gamma_{ij}$ will certainly be $(c''/d)n$-rich.  
Setting $c_3 = c''/d$ we are done.

\item We require the notion of the resultant of two polynomials to prove this;
for details see \cite{Cox05}.
Let $R(x,y)$ be the resultant with respect to $t$ 
(so considering $x,y$ as coefficients) 
of the two polynomials $x - \widehat{f}_i(t)$ and $y - \widehat{f}_j(t)$.
This is an irreducible polynomial of degree $\le 2d$ 
with the property that $R(x,y) = 0$ 
if and only if there is a $t$ 
such that $x = \widehat{f}_i(t)$ and $y = \widehat{f}_j(t)$;
in other words, $\gamma_{ij}$ is the algebraic curve $R(x,y) = 0$.
\end{enumerate}
\end{proof}

\subsection{Decomposing $\widehat{f}_i$}
\begin{lemma}
There is a subset $S$ of $c'n^{2\alpha -1}$ of the $\gamma_{ij}$ that all 
coincide as point sets, and such that the set $T$ of $\widehat{f}_i$ occurring 
in the first coordinate of a $\gamma_{ij}\in S$
has size $c_4n^{2\alpha -1}$.
\end{lemma}
\begin{proof}
Since the $\gamma_{ij}$ are irreducible and have degree $\leq 2d$, 
we can apply the Curve Lemma \ref{lem:curve}. 
Thus there exists $n_0$ such that for $n>n_0$, 
there can be at most $C_{CL}n$ distinct $c_3n$-rich curves on $C\times C$, 
so $c_3n^{2\alpha}/C_{CL}n = c'n^{2\alpha -1}$ of them must coincide.

Set $c_4 = c'/2d$. 
If fewer than $c_4n^{2\alpha -1}$ of the $\widehat{f}_i$ occurred among these 
coinciding $\gamma_{ij}$, 
then some $\widehat{f}_i$ would have to occur at least $d+1$ times, 
say in the first coordinate.
But if $(\widehat{f}_i, \widehat{f}_j)$ 
and $(\widehat{f}_i, \widehat{f}_k)$ coincide, 
then we must have $\widehat{f}_j = \widehat{f}_k$.
So we would have $d+1$ of the $\widehat{f}_i$ coinciding, 
hence also $d+1$ of the $f_i$, 
contradicting our Assumption after Lemma \ref{lem:nottoomanysame2D}.
\end{proof}

\begin{lemma}
   There are $c_4n^{2\alpha -1}$ $f_i$ with $$f_i(x) = p(a_i k(x) + b_i)$$ where $a_i, 
   b_i\in \mathbb{R}$ and $p\in \mathbb{R}[x]$.
\end{lemma}
\begin{proof}
   By the Reparametrization Lemma~\ref{lem:repar}, for each coinciding pair of 
   curves $\gamma_{ij}$ and $\gamma_{kl}$ from $S$, we can find $p$, 
   $\varphi_i$, and $\varphi_k$ such that $$\widehat{f}_i = 
   p\circ\varphi_i~~~\mathrm{and}~~~\widehat{f}_k = p\circ\varphi_k.$$
Hence we have such decompositions for each pair of the $\widehat{f}_i\in T$.

The $\widehat{f}_i$ were constructed so that 
any pair corresponding to an edge of $\widehat{E}$ has no nonlinear common inner function.
That implies that the $\varphi_i$ are linear, hence invertible, 
which allows us to assume that all $\widehat{f}_i\in T$ can be decomposed using the same $p$.
Indeed, if $\widehat{f}_i = p \circ \varphi_i = q\circ \phi_i$ and $\widehat{f}_j = q\circ \phi_j$, 
then $q = p\circ (\varphi_i \circ \phi_i^{-1})$, 
so we can write $\widehat{f}_j = p\circ (\varphi_i \circ \phi_i^{-1}\circ\phi_j)$;
by repeatedly modifying the $\varphi_i$ this way we can reach all $\widehat{f}_k\in T$.

Write $\varphi_i(t) = a_it + b_i$;
then for the $c_4n^{2\alpha-1}$ $f_i = \widehat{f}_i\circ k$ with $\widehat{f}_i\in T$ we have 
$f_i = p\circ\varphi_i\circ k$.
\end{proof}

\subsection{Proof of Theorem \ref{thm:ER2}}
At this point we will apply the Generalized Line Lemma \ref{lem:genLine} with $\beta =0$
to obtain Theorem \ref{thm:ER2}.
Then we need $2\alpha -1 = 2/3$, so we set $\alpha = 5/6$.

Note that the $c_4n^{2/3}$ lines $u = \varphi_i(t) = a_it + b_i$ live on 
$k(A)\times p^{-1}(D)$, which is essentially an $n\times n$ cartesian product 
(both sets might be smaller than $n$, but we can just add arbitrary points to 
fill them out).  They are $c_1n$-rich there, since otherwise the $f_i$ couldn't be 
$c_1n$-rich.

We conclude that either $d+1$ of the lines $u = \varphi_i(t)$ are parallel, 
or $d+1$ are concurrent.
Otherwise, by Lemma~\ref{lem:genLine} with $\beta = 0$ 
there would be fewer than $C_{GLL}n^{2/3}$ lines.  
But we can take $\tilde{c}$ in Theorem~\ref{thm:ER2} to be large enough so that $c_4 > C_{GLL}$. 
Indeed, one can easily check that each $c_i$ was an increasing unbounded function of $c_{i-1}$.

By Lemma \ref{parallelcase} below, if $d+1$ of the lines are parallel, 
then $f$ has the additive form $f(x,y) = p(k(x)+l(y))$.
By Lemma \ref{concurrentcase} below, if $d+1$ of the lines are concurrent, 
then $f$ has the multiplicative form $f(x,y) = p(k(x)\cdot l(y))$.
That finishes the proof of Theorem \ref{thm:ER2}.

\subsection{Proof of Theorem~\ref{thm:ERbest}} \label{subsec:ER3Croot}

We will now use Corollary \ref{cor:croot}, 
instead of Lemma \ref{lem:genLine} as above,  
which will result in Theorem~\ref{thm:ERbest}.  

We start with $\alpha = 1/2 + \varepsilon$.
Then we end up with $c_4n^{2\alpha -1} = c_4n^{2\varepsilon}$ lines $u = a_i t + b_i$
which are $c_1n$-rich on an $n\times n$ cartesian product.
Certainly $c_4n^{2\varepsilon} > 2(d+1)n^{\varepsilon'}$ for some $\varepsilon'>0$, 
so by Corollary \ref{cor:croot} with $p = q = d+1$ 
either $d+1$ of the lines are parallel or $d+1$ are concurrent.

By Lemma \ref{parallelcase} below, 
the parallel case would give the additive form for $f$, 
and Lemma \ref{concurrentcase} below,
the concurrent case would give the multiplicative form for $f$.
That finishes the proof of Theorem~\ref{thm:ERbest}.

\subsection{The parallel case}\label{subsec:ER3para}
\begin{lemma}\label{parallelcase}
If $d+1$ of the lines $\varphi_i$ are parallel, 
then there is a polynomial $l(y)$ such that $f(x,y) = p(k(x)+l(y))$.
\end{lemma}
\begin{proof}
The lines can be written as $\varphi_i(t) = at + b_i$, so by modifying $k$ 
we can write $f_i(x) = p(k(x) + b_i)$, for $d+1$ different $f_i$.
 We use the following two polynomial expansions of $f_i(x) = f(x,y_i) = p(k(x) + 
 b_i)$:
$$\sum_{l=0}^N v_l \cdot (k(x)+b_i)^l = \sum_{m=0}^Nw_m(y_i)\cdot k(x)^m.$$
The first is immediate from $p(k(x) + b_i)$; the second requires a little more 
thought. 

By Lemma~\ref{lem:alg}, there is a unique expansion of the polynomial $f$ of the 
form $f(x,y) = \sum_{l=0}^{D-1} c_l(k(x),y)x^l$, where $D = \deg k$.
By the same lemma, we have a unique expansion $f_i(x) = \sum_{l=0}^{D-1} 
d_l(k(x))x^l$, so that we have $$\sum_{l=0}^{D-1} c_l(k(x),y_i)x^l = 
\sum_{l=0}^{D-1} d_l(k(x))x^l~~\Rightarrow~~c_l(k(x), y_i) = d_l(k(x)).$$
But since $f_i(x) = p(k(x) + b_i)$, uniqueness implies that $d_l =0$ for $l>0$, 
hence $c_l(k(x),y_i) = 0$ for $l>0$.
Since we have this for $d+1$ different $i$, it follows that $c_l(k(x), y) = 0$ 
for $l>0$, so $f(x,y) = c_0(k(x),y)$, which means there is an expansion $f(x,y) 
= \sum w_m(y) k(x)^m$.
Now plugging in $y=y_i$ gives the required expansion.

Comparing the coefficients of $k(x)^{N-1}$ in the two expansions above, we get 
$$v_{N-1} + (N-1)v_N b_i = w_{N-1}(y_i),$$
which implies that $b_i = \frac{1}{(N-1)v_N} (w_{N-1}(y_i) - v_{N-1})$. If we 
now define the polynomial
$$l(y) = \frac{1}{(N-1)v_N} (w_{N-1}(y) - v_{N-1}),$$
we have that for $d+1$ of the $y_i$ (note that $v_l$ and $w_m$ do not depend on 
the choice of $y_i$)
$$f(x,y_i) = p(k(x)+l(y_i)).$$
Since the degree of $f$ is $d$, this implies that $f(x,y) = p(k(x)+l(y))$.
\end{proof}

\subsection{The concurrent case}\label{subsec:ER3conc}
\begin{lemma}\label{concurrentcase}
If $d+1$ of the lines $\varphi_i$ are concurrent,
then there are polynomials $P(t)$, $K(x)$ and $L(y)$ 
such that $$f(x,y) = P(K(x)\cdot L(y)).$$
\end{lemma}
\begin{proof}
The lines can be written as $\varphi_i(t) = a_it + b$, so by modifying $k$ 
we can write $f_i(x) = p(a_i\cdot k(x))$, for $d+1$ different $f_i$.
 We again use two polynomial expansions of $f_i(x) = f(x,y_i) =
p(a_i\cdot k(x))$:
$$\sum_{l=0}^N v_l \cdot (a_i\cdot k(x))^l =
\sum_{m=0}^Nw_m(y_i)\cdot k(x)^m.$$
Both are obtained in the same way as in the proof of 
Lemma~\ref{parallelcase}.

We cannot proceed exactly as before, since $a_i$ might occur
here only with exponents, and we cannot take a root of a polynomial.
But we can work around that as follows.
Define $M$ to be the greatest common divisor of all exponents $m$ for
which $w_m\neq 0$ in the second expansion;
then we can write $M$ as an integer linear combination of these $m$, say $M =
\sum \mu_m m$.
Comparing the coefficients of any $k(x)^m$ with $w_m\neq 0$ in the two
expansions above, we get
$$a_i^m = \frac{1}{v_m}w_m(y_i),$$
which tells us that
$$a_i^M = \prod (a_i^m)^{\mu_m} = L(y_i),$$
where $L(y)$ is a rational function.

If we define $P(s) = p(s^{1/M})$, or
equivalently $P(t^M) = p(t)$, then the definition of $M$ gives that $P(s)$ is a 
polynomial.
We also define $K(x) = k^M(x)$. Then
$$P(K(x)\cdot L(y_i)) = P(k^M(x)\cdot a_i^M) = p(k(x) \cdot
a_i) = f(x,y_i).$$
Since we have this for $d+1$ of the $y_i$, we get that $f(x,y) =
P(K(x) L(y))$.
This also tells us that $L(y)$ is in fact a polynomial, since
otherwise $f(x,y)$ could not be one.
\end{proof}

\section{Proof of Theorems~\ref{thm:main}, \ref{thm:ER3}, and \ref{thm:ER4}} 
\label{sec:er4d}

Suppose $w = f(x,y,z)$ contains $cn^{1+2\alpha}$ points of $A\times B\times 
C\times D$
and $|B|= |C| = n^{\alpha}$.
For Theorem \ref{thm:main} we have $\alpha = 1$; for the other two theorems
we will determine the right choice of $\alpha$ later.
Throughout we will use $d = \deg f$. All functions will be polynomials.
We will shorten or omit several of the proofs, 
because they are very similar to those in Section \ref{sec:erLess}.

\subsection{Constructing $\widehat{f}_{ij}$}
For each of the $n^{2\alpha}$ points $(y_i, z_j)\in B\times C$, 
we cut a fibre out of the solid:
$$w = f_{ij}(x) = f(x, y_i, z_j).$$
\begin{lemma}\label{nottoomanyidentical}
If at least $2dn^\alpha$ of the $f_{ij}$ are identical, then $f(x,y,z) = q(x)$.\\
In particular, the conclusion of 
Theorems \ref{thm:main}, \ref{thm:ER3}, and \ref{thm:ER4} holds.
\end{lemma}
\begin{proof}
Suppose that $f_{ij}(x) = q(x)$ at least $2dn^\alpha$ times.  
Then for $F(y,z) = f(x,y,z) - q(x)$ and $K = \mathbb{R}(x)$, 
the Vanishing Lemma \ref{lem:vanish} with $b = c = n^\alpha$ gives $F(y,z) = 0$.
\end{proof}\noindent
{\bf Assumption:} Throughout the rest of this proof we will assume that fewer 
than $2dn^\alpha$ of the $f_{ij}$ are identical.
\bigskip


Let $c_1 = c/2$.  Then at least $c_1n^{2\alpha}$ of the $f_{ij}$ are 
$c_1n$-rich on $A\times D$.  Otherwise $w=f(x,y,z)$ would contain fewer than 
$cn^{1+2\alpha}$ points of $A\times B\times C\times D$.

We construct a graph $G$ with the $c_1n^{2\alpha}$ $f_{ij}$ as vertices 
and edge set $E$ consisting of the pairs $(f_{ij}, f_{kl})$.
\begin{lemma}
 There is a subgraph of $G$ with edge set $\widehat{E}\subset E$ of size 
 $|\widehat{E}|\geq c_2n^{4\alpha}$, such that the following holds.
There is a polynomial $k(x)$ such that 
for all $(f_{ij}, f_{kl})\in \widehat{E}$ we can write
$$f_{ij} = \widehat{f}_{ij}\circ k,$$ $$f_{kl} = \widehat{f}_{kl}\circ k,$$
and $\widehat{f}_{ij}$ and $\widehat{f}_{kl}$ share no non-linear inner function.\\
The $\widehat{f}_{ij}$ are also $c_2n$-rich on $k(A)\times D$.
\end{lemma}


\subsection{Constructing $\gamma_{ijkl}$}
For the $c_2n^{4\alpha}$ pairs $\widehat{f}_{ij}, \widehat{f}_{kl}$ 
for which $(f_{ij},f_{kl})\in \widehat{E}$ 
we construct the curves
$$\widehat{\gamma}_{ijkl}(t) 
= \left(\widehat{f}_{ij}(t), \widehat{f}_{kl}(t)\right),$$
\begin{lemma}\text{ }
\begin{enumerate}
\item At least $c_3n^{4\alpha}$ of the $\gamma_{ijkl}$ 
are $c_3n$-rich on $D\times D$.
\item Each $\gamma_{ijkl}$ is an irreducible algebraic curve of degree at most $2d$.
\end{enumerate}
\end{lemma}
\begin{proof}
\begin{enumerate}
\item

We define a bipartite graph with vertex set $E = k(A) \cup \{\widehat{f}_{ij}\}$, 
and we connect $t\in k(A)$ with $f_{ij}$ if $f_{ij}(t)\in D$.
Then this graph has $m = c_2^{3/2}n^{1+2\alpha}$ edges.  
We count the 2-paths:
$$\#P_2 = \sum_{x\in k(A)} \binom{d(x)}{2}\geq 
|k(A)|\binom{m/|k(A)|}{2}\geq c' n^{1+4\alpha}.$$
Hence at least $c''n^{4\alpha}$ pairs $\widehat{f}_{ij},\widehat{f}_{kl}$ 
share $c''n$ common neighbors $t$ in this graph.
This implies that if $c_3 = c''/d$ then $c_3n^{4\alpha}$ of the $\gamma_{ijkl}$
have at least $c_3n$ point in $D\times D$.


\end{enumerate}
\end{proof}

\subsection{Decomposing $\widehat{f}_{ij}$}
\begin{lemma}
There is a subset of $c_4n^{4\alpha-1}$ of the $\gamma_{ijkl}$ that all coincide, 
and such that $c_4n^{3\alpha-1}$ of the $\widehat{f}_{ij}$ occur in these $\gamma_{ijkl}$.
\end{lemma}
\begin{proof}
By the Curve Lemma, for $n>n_0$, there can be at most $C_{CL}n$ distinct $c_3n$-rich 
curves on $D\times D$, so $c'n^{4\alpha -1}$ must coincide.
Setting $c_4 = c'/2dn$ gives that at least $c_4n^{3\alpha -1}$ of the $\widehat{f}_{ij}$ occur.

\end{proof}
\begin{lemma}
There are $c_4n^{3\alpha-1}$ pairs $(i,j)$ for which $$f_{ij}(x) = p(a_{ij} k(x) + 
b_{ij})$$ where $a_{ij}, b_{ij}\in \mathbb{R}$ and 
$p\in\mathbb{R}[x]$.
\end{lemma}
\begin{proof}
For each coinciding pair of curves $\gamma_{ijkl}$ and 
$\gamma_{abcd}$, by the Reparametrization Lemma we can write
$$\widehat{f}_{ij} = p\circ\varphi_{ij}~~~
\text{and}~~~\widehat{f}_{ab} = p\circ\varphi_{ab}.$$
By construction of the $\widehat{f}_{ij}$, the $\varphi_{ij}$ must be linear, 
which allows us to assume that all pairs use the same $p$.
Write $\varphi_{ij}(t) = a_{ij}t + b_{ij}$;
then for the $c_4n^{3\alpha-1}$ corresponding $f_{ij}$ we have 
$f_{ij} = p\circ\varphi_{ij}\circ k$.
\end{proof}

\subsection{Proof of Theorem \ref{thm:main}}
Here we set $\alpha = 1$, 
so we have $c_4n^2$ rich lines $u = \varphi_{ij}(t) = a_{ij}t + b_{ij}$
that are rich on the (essentially) $n\times n$ cartesian product $k(A)\times p^{-1}(D)$.

We claim that either $c_5n^2$ of the lines $u = \varphi_{ij}(t)$ are parallel, 
or $c_5n^2$ are concurrent, counting multiplicities.
By the Szemer\'edi-Trotter Theorem (\ref{thm:szem}), 
at most $C_{ST}n$ of the lines are distinct.  
By our Assumption after Lemma \ref{nottoomanyidentical}, 
fewer than $2dn$ are identical.
This implies that for some $c'$ we can split the lines 
into $c'n$ classes of size at least $c'n$, 
such that within each class the lines are identical, 
and between the classes the lines are distinct. 

We take a representative of each class and 
apply the Line Lemma \ref{lem:line} to these $c'n$ representatives, 
telling us that $c''n$ are parallel or $c''n$ are concurrent.
Taking all of the corresponding classes together gives $(c''\cdot c')n^2$ lines 
that are all parallel or all concurrent.

By Lemma \ref{parallelcase4d} below, we only need $2dn$ lines parallel, 
to show that $f$ has the additive form $f(x,y,z) = p(k(x)+l(y)+m(z))$,
so $c''c'n^2$ will certainly suffice.
Similarly, by Lemma \ref{parallelcase4d}, if $c''c'n^2$ of the lines are 
concurrent, then $f$ has the multiplicative form $f(x,y,z) = p(k(x)\cdot 
l(y)\cdot m(z))$.
That finishes the proof of Theorem \ref{thm:main}.

\subsection{Proof of Theorem \ref{thm:ER3}}
We have $c_5n^{3\alpha-1}$ $c_5n$-rich lines, 
for an $\alpha$ to be determined below.  
Many of these lines may coincide,
so we split them into $n^\beta$ classes of coinciding lines.  
The average size of a class is then $c_5n^{3\alpha - 1 - \beta}$,
so for some $c'>0$ and $\varepsilon>0$ 
we can find a subset of $c' n^\beta$ classes 
that all have size at least 
$c'n^{3\alpha - 1 - \beta - \varepsilon}$.

To apply Lemmas \ref{parallelcase4d} and \ref{concurrentcase4d}
and finish the proof, 
we will need $2dn^\alpha$ lines that are all parallel or concurrent.
To obtain these we need 
$\frac{2d}{c'}n^{\alpha - (3\alpha - 1 - \beta - \varepsilon)} 
= \frac{2d}{c'}n^{1+\beta+\varepsilon -2\alpha}$ 
representatives of the coinciding classes 
that are all parallel or concurrent, 
since each class has size at least $c'n^{3\alpha - 1 - \beta - \varepsilon}$.

To get these representatives using Lemma \ref{lem:genLine}, 
we need 
$$c'n^{\beta} \geq 
\frac{2d}{c'}n^{2/3+ (1+\beta+\varepsilon -2\alpha)/3},$$
for which it suffices to have
$3\beta-\varepsilon \geq 2 +1 +\beta +\varepsilon -2\alpha$, 
or $\beta \geq 3/2 +\varepsilon - \alpha$.

On the other hand, if any of the $n^\beta$ classes 
contains at least $2dn^{\alpha}$ lines,
then also $2dn^\alpha$ of the $f_{ij}$ would be identical, 
contradicting our assumption after Lemma \ref{nottoomanyidentical}.
Hence all classes are smaller than $2dn^\alpha$,
which implies that 
$$n^\beta \geq \frac{c_5}{2d}n^{2\alpha-1},$$ 
hence $\beta \geq 2\alpha -1-\varepsilon$.	

The second inequality for $\beta$ will imply the first if
$$2\alpha - 1 -\varepsilon\geq 3/2 +\varepsilon - \alpha,$$
hence $\alpha = 5/6 + \varepsilon$ will do.

\subsection{Proof of Theorem \ref{thm:ER4}}
For Theorem \ref{thm:ER4}, we do the same as for Theorem \ref{thm:ER3}, 
except that instead of Lemma \ref{lem:genLine} 
we apply Corollary \ref{cor:croot}.
To get the right number of parallel or concurrent lines,
we set $p=q=2dn^{1+\beta+\varepsilon -2\alpha}$ in the Corollary,
so we require 
$$c'n^\beta > (p+q)n^{\varepsilon'} = 4dn^{1+\beta+\varepsilon -2\alpha+\varepsilon'}$$
for some $\varepsilon'$.
That will hold if $\beta > 1+\beta+\varepsilon -2\alpha+\varepsilon'$,
or $\alpha \geq 1/2 + \varepsilon/2 + \varepsilon'/2$, 
which is satisfied for $\varepsilon' = \varepsilon$ 
and $\alpha = 1/2+\varepsilon$ as in Theorem \ref{thm:ER4}.

\subsection{The parallel case}\label{subsec:ER4para}
\begin{lemma}\label{parallelcase4d}
If $2dn^\alpha$ of the lines $\varphi_{ij}$ are parallel,
then there is a polynomial $r(y,z)$ such that $f(x,y,z) = p(k(x)+r(y,z))$.
\end{lemma}
\begin{proof}
We can write $f_{ij}(x) = p(k(x) + b_{ij})$.
 We use the following two polynomial expansions of 
$f_{ij}(x) = f(x,y_i,z_j) =  p(k(x) + b_{ij})$:
$$\sum_{l=0}^N v_l \cdot (k(x)+b_{ij})^l = \sum_{m=0}^Nw_m(y_i,z_j)\cdot k(x)^m.$$
The first is immediate from $p(k(x) + b_{ij})$; the second requires a little more thought. 

By Lemma~\ref{lem:alg}, there is a unique expansion of the polynomial $f$ of the 
form $f(x,y,z) = \sum_{l=0}^{D-1} c_l(k(x),y, z)x^l$, where $D = \deg k$.
By the same lemma, we have a unique expansion $f_{ij}(x) = \sum_{l=0}^{D-1} 
d_l(k(x))x^l$, so that we have $$\sum_{l=0}^{D-1} c_l(k(x),y_i, z_j)x^l = 
\sum_{l=0}^{D-1} d_l(k(x))x^l~~\Rightarrow~~c_l(k(x), y_i,z_j) = d_l(k(x)).$$
But since $f_{ij}(x) = p(k(x) + b_{ij})$, uniqueness implies that $d_l =0$ for 
$l>0$, hence $c_l(k(x),y_i,z_j) = 0$ for $l>0$. 
We have this for every $y_i,z_j$ such that $\varphi_{ij}$ is one of the parallel lines.

Then we have $2dn^\alpha$ zeroes of $c_l(k(x),y,z)$, 
so applying the Vanishing Lemma with $|B|=|C|=n^\alpha$  
gives $c_l(k(x), y,z) = 0$ for $l>0$.
Thus $f(x,y,z) = c_0(k(x),y,z)$, 
which means there is an expansion $f(x,y,z) = \sum w_m(y,z) k(x)^m$.
Now plugging in $y=y_i, z = z_j$ gives the expansion required above.

Comparing the coefficients of $k(x)^{N-1}$ in the two expansions above, we get 
$$v_{N-1} + (N-1)v_N b_{ij} = w_{N-1}(y_i,z_j),$$
which implies that $b_{ij} = \frac{1}{(N-1)v_N} (w_{N-1}(y_i,z_j) - v_{N-1})$. If we now define the polynomial
$$r(y,z) = \frac{1}{(N-1)v_N} (w_{N-1}(y,z) - v_{N-1}),$$
we have that for our $2dn^\alpha$ pairs $(y_i, z_j)$ (note that $v_l$ and $w_m$ do 
not depend on the choice of pair)
$$f(x,y_i,z_j) = p(k(x)+r(y_i,z_j)).$$
By the Vanishing Lemma with $|B|=|C|=n^\alpha$, 
applied to $F(y,z) = f(x,y,z) - p(k(x)+r(y,z))$ over $K = \mathbb{R}(x)$,
we get the desired equality $f(x,y,z) = p(k(x)+r(y,z))$.
\end{proof}\noindent

\begin{lemma}
 There are polynomials $l$ and $m$ such that 
$$f(x,y,z) = p(k(x) + l(y) + m(z)).$$
\end{lemma}
\begin{proof}
By applying the above with the roles of $x$ and $y$ swapped, we can also write $f(x,y,z) = P(K(y) + R(x,z))$. 
Then we calculate the quotient $f_x/f_y$ (using the notation $f_x = \partial f/\partial x)$ for both forms,
$$\frac{f_x}{f_y} = \frac{k'(x)}{r_y(y,z)} = \frac{R_x(x,z)}{K'(y)},$$
which tells us that $r_y(y,z)$ (and $R_x(x,z)$) is independent of $z$.  
Integrating with respect to $y$ then gives that $r(y,z) = l(y) + m(z)$, which 
proves our claim.
\end{proof}

\subsection{The concurrent case}\label{subsec:ER4conc}
\begin{lemma}\label{concurrentcase4d}
If $2dn^\alpha$ of the lines $\varphi_{ij}$ are concurrent,
There are polynomials $P(t)$, $K(x)$ and $R(y,z)$ such that 
$$f(x,y,z) = P(K(x)\cdot R(y,z)).$$
\end{lemma}
\begin{proof}
We can write $f_{ij}(x) = p(a_{ij}\cdot k(x))$.
We again use two polynomial expansions of 
$f_{ij}(x) = f(x,y_i,z_j) =p(a_{ij}\cdot k(x))$:
$$\sum_{l=0}^N v_l \cdot (a_{ij}\cdot k(x))^l =
\sum_{m=0}^Nw_m(y_i,z_j)\cdot k(x)^m.$$
Both are obtained in the same way as in the proof of Claim
\ref{parallelcase4d}.

We cannot proceed exactly as before, 
since $a_{ij}$ might only occur here with exponents, 
and we cannot take a root of a polynomial.
But we can work around that as follows.
Define $M$ to be the greatest common divisor of all exponents $m$ for
which $w_m\neq 0$ in the second expansion;
then we can write $M$ as an integer linear combination of these $m$, 
say $M =\sum \mu_m m$.
Comparing the coefficients of any $k(x)^m$ 
with $w_m\neq 0$ in the two expansions above, we get
$$a_{ij}^m = \frac{1}{v_m}w_m(y_i,z_j),$$
which tells us that
$$a_{ij}^M = \prod (a_{ij}^m)^{\mu_m} = R(y_i,z_j),$$
where $R(y,z)$ is a rational function.

If we define $P(s) = p(s^{1/M})$, or equivalently $P(t^M) = p(t)$, 
then the definition of $M$ gives that $P(s)$ is a polynomial.
We also define $K(x) = k^M(x)$. 
Then for each of the $2dn^\alpha$ pairs $y_i,z_j$ we have
$$P(K(x)\cdot R(y_i,z_j)) = P(k^M(x)\cdot a_{ij}^M) = p(k(x) \cdot a_{ij}) = f(x,y_i,z_j).$$
Applying the Vanishing Lemma with $|B|=|C|=n^\alpha$ over $\mathbb{R}(x)$ 
to the numerator of $f(x,y,z) - P(K(x)R(y,z))$, 
we get that $f(x,y,z) = P(K(x) R(y,z))$.
This also tells us that $R(y,z)$ is in fact a polynomial, 
since otherwise $f(x,y,z)$ could not be one.
\end{proof}
\noindent
\begin{lemma}
 There are polynomials $L$ and $M$ such that $f(x,y,z) = P(K(x) \cdot L(y) \cdot  M(z))$.
\end{lemma}
\begin{proof}
By applying the above with the roles of $x$ and $y$ swapped, we can also write 
$f(x,y,z) = P^*(K^*(y)\cdot R^*(x,z))$.  Then we calculate the quotient 
$f_x/f_z$ for both forms,
$$\frac{f_x}{f_z} = \frac{K'(x)R(y,z)}{K(x)R_z(y,z)} = 
\frac{R^*_x(x,z)}{R^*_z(x,z)},$$
which tells us that $$\frac{R_z(y,z)}{R(y,z)} = \frac{\partial}{\partial 
z}\log(R(y,z))$$ is independent of $y$. Integrating we get that $\log(R(y,z)) = 
\lambda(y) + \mu(z)$, hence $$R(y,z) = e^{\lambda(y)}\cdot e^{\mu(z)} = 
L(y)M(z),$$ which also implies that $L(y)$ and $M(z)$ are polynomials, as 
desired.
\end{proof}
\noindent
This finishes the proof.


\section{Applications and Limitations} \label{sec:apps}

In this section we give some applications and limitations of the main results.  
We start by giving a simple condition to check whether a function has the 
required additive or multiplicative form required in the main results.  Then we 
give a proof of our variant of Purdy's conjecture.  Finally we give a 
construction using parabolas that shows that the exponents in 
Theorem~\ref{thm:ER4} cannot be improved significantly.

\subsection{How to check if a function is additive or multiplicative}\label{check}

Given a differentiable function $f(x,y):\mathbb{R}^2\to\mathbb{R}$, we define 
\[q_f(x,y) = \frac{\partial^2}{\partial x \partial y}\log\biggl[\frac{\partial 
f/\partial x}{\partial f/\partial y}\biggr].\]

Suppose $f$ is of the form $f(x,y)=p(k(x)+l(y))$ or $f(x,y)=p(k(x)l(y))$, where 
$p,k$ and $l$ are nonconstant.  Then one can check that \[q_f(x,y) = 0\] 
identically.

So, if we have a differentiable function $f:\mathbb{R}^2\to\mathbb{R}$, and 
$q_f$ is not identically zero, then we know that the function does not have the 
additive or multiplicative form.
The converse of this result also holds, although we do not need that fact here.  

A similar condition holds for functions $f$ of the form 
$f(x,y,z)=p(k(x)+l(y)+m(z))$ or $f(x,y,z)=p(k(x)l(y)m(z))$.  If we define 
\[q_f(x,y,z) = \frac{\partial^2}{\partial x \partial y}\log\biggl[\frac{\partial 
f/\partial x}{\partial f/\partial y}\biggr].\]  Then $q_f(x,y,z) = 0$.

Notice that with $f$ in the form above, \[\frac{\partial f/\partial x}{\partial 
f/\partial y}=\frac{k'(x)}{l'(y)}\quad\mathrm{or}\quad\frac{\partial f/\partial 
x}{\partial f/\partial 
y}=\frac{k'(x)l(y)m(z)}{k(x)l'(y)m(z)}=\frac{k'(x)l(y)}{k(x)l'(y)}\] is 
independent of $z$.  This provides another way of checking whether a function 
does not have the additive or multiplicative form.

Similar conditions could be checked using partial derivatives with respect to 
$z$.  If $f(x,y,z)=p(k(x)+l(y)+m(z))$ or $f(x,y,z)=p(k(x)l(y)m(z))$ we get 
\[r_f(x,z) = \frac{\partial^2}{\partial x \partial z}\log\biggl[\frac{\partial 
f/\partial x}{\partial f/\partial z}\biggr] = 0\] and \[s_f(y,z) = 
   \frac{\partial^2}{\partial y \partial z}\log\biggl[\frac{\partial f/\partial 
   y}{\partial f/\partial z}\biggr] = 0.\]

Note that in this case the converse does not hold.  In the example in 
Section~\ref{subsec:parab} below $q_f=0$, $r_f=0$ and $s_f=0$, but $f$ does not 
have the required decomposition.

\subsection{On a conjecture of Purdy} \label{subsec:purdy}
The following theorem was conjectured by G. Purdy in \cite{Bras06} and proved by Elekes 
and R\'onyai in \cite{Elek00}. 
We will use the notation $D(P,Q) = \{d(p,q) : p\in P, q\in Q\}$
for the set of distances between two point sets.
\begin{theorem}
For all $c$ there is an $n_0$ such that for $n>n_0$ the following holds
for any two lines $\ell_1$ and $\ell_2$ in $\mathbb{R}^2$ 
and sets $P_i$ of $n$ points on $\ell_i$. 

If $|D(P_1,P_2)|< cn$ then the two lines are parallel or orthogonal.
\end{theorem}

Using Theorem~\ref{thm:ERbest} (or Theorem~\ref{thm:ER2}) we can extend it to 
the asymmetric case when we have fewer points on one of the lines.  The proof is 
similar to that in \cite{Elek00}.
\begin{theorem}
For every $c>0$ and $\varepsilon>0$ there is an $n_0$ such that for $n>n_0$ the 
following holds
for any two lines $\ell_1$ and $\ell_2$ in $\mathbb{R}^2$,
$P_1$ a set of $n$ points on $\ell_1$,
and $P_2$ a set of $n^{1/2+\varepsilon}$ points on $\ell_2$.
  
If $|D(P_1,P_2)|< cn$ then the two lines are parallel or orthogonal.
\end{theorem}
\begin{proof}
Parameterize $l_1$ by $x_1$ and $l_2$ by $x_2$, 
and let $X_1$ and $X_2$ represent $P_1$ and $P_2$ in this parameterization.
Then the condition on the distances means by the Law of Cosines that 
the polynomial $f(x_1,x_2)=x_1^2+2\lambda x_1x_2 + x_2^2$ 
assumes $<cn$ values on $X_1\times X_2$.


Then $z = f(x_1,x_2)$ contains $>c'n^{3/2 +\varepsilon}$ points of the cartesian 
product $X_1\times X_2\times E$ where $E = \{a^2:a\in D(P_1,P_2)\}$.  By 
Theorem~\ref{thm:ERbest}, this implies that $f$ has the additive or 
multiplicative form.
Thus $q_f$, as defined in Section \ref{check}, should be identically zero.  
A quick calculation shows that this is only possible if $\lambda = -1,0,$ or $1$,
which means that the angle between the lines is $0$ or $\pi/2$.
Therefore the lines are parallel or concurrent.
\end{proof}

\subsection{Limits on the asymmetry of the cartesian product} 
\label{subsec:parab}

In this section we show that Theorem~\ref{thm:ER4} is near-optimal. 
We will use the notation $[a,b] = \{a, a+1, \ldots, b-1, b\}$.

Consider $$f(x,y,z) = x + (y-z)^2,$$
and let $A = D =  [1,k^2]$ and $B = C =  [1,k]$ for an even integer $k$.
If we set $n = k^2$, then $|A| = |D| = n$ and $|B| = |C| = n^{1/2}$.
We can think of the solid $w = f(x,y,z)$ as consisting of translates of the 
parabola $w = y^2$ from the $wy$-plane.

We have $x+(y-z)^2 \in D$ when (for instance) $$x\in [1,k^2/2],\quad y \in 
[1,k/2] \quad \text{and}\quad z\in [1,k/2].$$
Then the solid $w = f(x,y,z)$ contains at least 
$\frac{1}{8}k^4 = \frac{1}{8} n^2$ points
of $A\times B\times C\times D$.

But the function $f(x,y,z)=(y-z)^2+x$ does not have one of the forms $p(k(x)+l(y)+m(z))$ 
or $P(K(x)L(y)M(z))$.  Note that $q_f = 0$, $r_f = 0$ and $s_f = 0$, so we cannot 
use the method above to show that $f$ does not have the additive or 
multiplicative form.  Instead we consider a degree argument.

Suppose $f(x,y,z) = P(K(x)L(y)M(z))$.  
Since each of $P,K,L$ and $M$ must have degree at least one, 
we would have $\deg f\ge 3$, a contradiction.  
So $f$ does not have the multiplicative form.  

Now suppose that $f(x,y,z)=p(k(x)+l(y)+m(z))$.  
Then $p,k,l$ and $m$ have degree at least one and at most two.  
If $\deg p = 2$ then $\deg k = 1$, 
implying $f$ has a term of the form $cx^2$, which it doesn't.
If $\deg p = 1$, then $f$ couldn't contain the term $-2yz$.
So $f$ does not have the additive form either.

Therefore the graph of $w=f(x,y,z)=x+(y-z)^2$ contains many points of 
$A\times B\times C\times D$,
but $f$ cannot be written in the additive or multiplicative form.
Hence any extension of Theorem~\ref{thm:main} with $|B|=|C|$ 
would have to have $|B|=|C|\ge cn^{1/2}$ for some constant $c>0$.  
Theorem~\ref{thm:ER4} supposes $|B|=|C|=n^{1/2+\varepsilon}$ for some $\varepsilon>0$, 
so that condition cannot be improved significantly.

\bibliographystyle{plain}
\bibliography{references}
\end{document}